\documentclass{birkjour}
 \newtheorem{theorem}{Theorem}[section]
 \newtheorem{corollary}[theorem]{Corollary}
 
 \newtheorem{proposition}[theorem]{Proposition}
 \theoremstyle{definition}
 \newtheorem{definition}[theorem]{Definition}
 \theoremstyle{remark}
 \newtheorem{rem}[theorem]{\bf{Remark}}
 \newtheorem*{ack}{\bf{Acknowledgement}}
 
 \numberwithin{equation}{section}

\begin{document}

%
%
%
%
%
%
%
%
%

\title[New curvature tensors along Riemannian submersion]
{New curvature tensors along Riemannian submersions}

\author[M. A. Akyol]{Mehmet Akif Akyol}
\address{Bingol University\\ Faculty of Arts and Sciences,\\ Department of Mathematics\\ 12000, Bing\"{o}l, Turkey}
\email{mehmetakifakyol@bingol.edu.tr}
\author[G. Ayar]{G\"{u}lhan Ayar}
\address{Karaman Mehmet Bey University\\ Department of Mathematics \\ 70000, Karaman, Turkey}
\email{gulhanayar@kmu.edu.tr}






\date{January 1, 2004}

\begin{abstract}
In 1966, B. O'Neill [The fundamental equations of a submersion, Michigan Math. J., Volume 13, Issue 4 (1966), 459-469.] obtained some fundamental equations and curvature relations between the total space, the base space and the fibres of a submersion.  In the present paper, we define new curvature tensors  along Riemannian submersions such as Weyl projective curvature tensor, concircular curvature tensor, conharmonic curvature tensor, conformal curvature tensor and $M-$projective curvature tensor, respectively. Finally, we obtain some results in case of the total space of Riemannian submersions has umbilical fibres for any curvature tensors mentioned by the above. 
\end{abstract}

\maketitle

\section{Introduction and Preliminaries}

In differential geometry, an important tool to define the curvature of $n-$ dimensional spaces (such as Riemannian manifolds) is the Riemannian curvature tensor. The tensor has played an important role both general relativty and gravity. In this manner, Mishra in \cite{mishra}  defined some new curvature tensors on Riemannian manifolds such as concircular curvature tensor, conharmonic curvature tensor, conformal curvature tensor, respectively. Taking into account the paper of Mishra, Pokhariyal and Mishra defined the Wely projective curvature tensor on Riemannian manifolds \cite{PM}. Afterwards, Ojha defined $M-$ projective curvature tensor \cite{OJ}. 

Curvature tensors also play an important role in physics. Relevance in physics of the tensors considered in our work and some of the important studies that focus on the geometric properties of curvature tensors are; in 1988 conditions of conharmonic curvature tensor of Kaehler hypersurfaces in complex space forms has been analysed by M. Doric et al. \cite{doric1}. By the way, the relativistic significance of concircular curvature tensor has been studied by Ahsan \cite{Ahsan1} and this tensor has been explored in that study. Finaly in 2018  based on Einstein’s geodesic postulate, projectively related connections on a space-time manifold and the closely related Weyl projective tensor has been examined detailed by G. Hall \cite{Hall1}.

Riemannian submersion appears to have been studied and its differential geometry has first defined by  O'Neill 1966 and  Gray 1967 \cite{O}. We note that Riemannian submersions have been studied widely not only in mathematics, but also in theoretical pyhsics because of their applications in the Yang-Mills theory, Kaluza Klein theory, super gravity, relativity and superstring theories (see \cite{BL1}, \cite{BL}, \cite{IV}, \cite{IV1}, \cite{M}, \cite{W1}). Most of the studies related to Riemannian submersion can be found in the books (\cite{FIP}, \cite{Sahin}). In 1966, B. O'Neill has defined a paper related to some fundamental equations of a submersion. In that paper, he has given some curvature relations on Riemannian submersions. 

In this study, in addition to the curvature relations previously defined on Riemannian submersion, we investigate new curvature tensors on a Riemannian submersion and the curvature properties of these tensors.
In the present paper, in the first part of our study, the basic definitions and theorems that we will use throughout the paper are given. In sections 2-6 include the Weyl projective curvature tensor, concircular curvature tensor, conharmonic curvature tensor, conformal curvature tensor and $M$-projective curvature tensor relations  for a Riemannian submersion respectively.  Also various results are obtained by examining the conditions for having total umbilical fibers.

Now, we will give the basic definitions and theorems without proofs that we will use throughout the paper.

\begin{definition}\rm
Let $(M,g)$ and $(N,g_{\text{\tiny$N$}})$ be Riemannian manifolds,
where-$dim(M)>dim(N)$. A surjective mapping
$\pi:(M,g)\rightarrow(N,g_{N})$ is called a \emph{Riemannian
	submersion}
\cite{O} if:\\
\textbf{(S1)}\quad The rank of $\pi$ equals $dim(N)$.\\
In this case, for each $q\in N$, $\pi^{-1}(q)=\pi_{q}^{-1}$ is a $k$-dimensional
submanifold of $M$ and called a \emph{fiber}, where $k=dim(M)-dim(N).$
A vector field on $M$ is called \emph{vertical} (resp.
\emph{horizontal}) if it is always tangent (resp. orthogonal) to
fibers. A vector field $X$ on $M$ is called \emph{basic} if $X$ is
horizontal and $\pi$-related to a vector field $X_{*}$ on $N,$ i.e. ,
$\pi_{*}(X_{p})=X_{*\pi(p)}$ for all $p\in M,$ where $\pi_{*}$ is derivative or differential map of $\pi.$
We will denote by $\mathcal{V}$ and $\mathcal{H}$ the projections on the vertical
distribution $ker\pi_{*}$, and the horizontal distribution
$ker\pi_{*}^{\bot},$ respectively. As usual, the manifold $(M,g)$ is called \emph{total manifold} and
the manifold $(N,g_{N})$ is called \emph{base manifold} of the submersion $\pi:(M,g)\rightarrow(N,g_{N})$.\\
\textbf{(S2)}\quad $\pi_{*}$ preserves the lengths of the horizontal vectors.\\
This condition is equivalent to say that the derivative map $\pi_{*}$ of $\pi$, restricted to $ker\pi_{*}^{\bot},$ is a linear
isometry.
\end{definition}
If $X$ and $Y$ are the basic vector fields, $\pi$-related to $X_{N}, Y_{N}$, we have the following facts:
\begin{enumerate}
	\item{} $g(X,Y)=g_{N}(X_{N},Y_{N})\circ\pi$,\\
	\item{} $h[X,Y]$ is the basic vector field $\pi$-related to $[X_{N},Y_{N}]$,\\
	\item{} $h(\nabla_{X}Y)$ is the basic vector field $\pi$-related to ${\nabla^{N}}_{{X}_{N}}Y_{N}$, \\
\end{enumerate}
for any vertical vector field $V$, $[X,V]$ is the vertical.\\
\indent
The geometry of Riemannian
submersions is characterized by O'Neill's tensors $\mathcal{T}$ and
$\mathcal{A}$, defined as follows:

\begin{equation}\label{e1}
\mathcal{T}_{E}F=v\nabla_{vE}hF+h\nabla_{vE}vF,
\end{equation}
\begin{equation}\label{e2}
\mathcal{A}_{E}F=v\nabla_{hE}hF+h\nabla_{hE}vF
\end{equation}
for any vector fields $E$ and $F$ on $M,$ where $\nabla$ is the
Levi-Civita connection, $v$ and  $h$ are orthogonal projections on vertical and horizontal spaces, respectively.

%


We now recall the following curvature relations for a Riemannian submersion from \cite{FIP} and \cite{O}.
\begin{theorem}\label{thm1}
	$(M, g)$ and $(G, g^\prime)$ Riemannian manifolds,
	$$\pi: (M, g) \to (G, g^\prime)$$
	a Riemannian submersion and $R^M$, $R^G$ and $\hat{R}$ be Riemannian curvature tensors of $M,G $ and $(\pi^{-1} (x), \hat{g}_x)$ fibre respectively. In this case, there are the following equations for any $U, V, W, F \in \chi^v (M)$ and $X, Y, Z, H \in \chi^h (M)$
	\begin{align}\label{g2}
	g(R^M(X,Y)Z,H) &= g({R}^G(X,Y)Z,H) + 2g(A_X Y, A_Z H)\nonumber \\
	&- g(A_Y Z, A_X H)+ g(A_X Z, A_Y H),
	\end{align}
	\begin{align}\label{g3}
	g(R^M(X,Y)Z,V)&= -g((\nabla_Z A)_X Y, V) - g(A_X Y, T_V Z) \nonumber\\
	&+g(A_Y Z, T_V X)- g(A_X Z, T_V Y),
	\end{align}
	\begin{align}\label{g4}
	g(R^M(X,Y)V,W) &= g((\nabla_V A)_X Y,W)- g((\nabla_W A)_X Y,V)\nonumber \\
	&+ g(A_X V A_Y W)+ g(A_X W, A_Y V) \nonumber\\
	&- g(T_V X, T_W Y)+ g(T_W X, T_V Y),
	\end{align}
	\begin{align}\label{g5}
	g(R^M(X,V)Y,W) &= g((\nabla_X T)_V W,Y)- g((\nabla_V A)_X Y,W)\nonumber \\
	&- g(T_V X, T_W Y)+ g(A_X V, A_Y W),
	\end{align}
	\begin{align}\label{g6}
	g(R^M(U,V)W,X)&=g((\nabla_U T)_V W, X) - g((\nabla_V T)_U W, X)
	\end{align}
	and
	\begin{align}\label{g1}
	g(R^M(U,V)W,F) &=g(\hat{R}(U,V)W,F) + g(T_U W, T_V F)- g(T_V W, T_U F).
	\end{align}
\end{theorem}

\begin{definition}\rm\cite{FIP}
	Let $(M, g)$ be a Riemannian manifold and a local orthonormal frame of the vertical distribution $\nu$ is $\{U_j\}_{1\leq j\leq r} $.
	Then $N$, the horizontal vector field on $(M, g)$ is locally defined by
	$$N = \sum_{j=1}^{r}  T_{U_j}U_j.$$
\end{definition}
\begin{proposition}\label{pro1}
	Let  $(M, g)$ and $(G, g^\prime)$ Riemannian manifolds,
	$$\pi: (M, g)\to (G, g^\prime) $$
	a Riemannian submersion
	and $\{X_i, U_j\}$ be  a $\pi$-compatible frame.
	
	In this case, for any $U, V \in \chi^v (M)$ and $X, Y \in  \chi^h (M)$, the Ricci tensor $S^M$ satisfies the following equations \cite{FIP}:
	
	\begin{align}\label{S1}
	(i) \,\,\,\,\,\,  S^M(U,V) &= \hat{S}(U,V) -g(N,T_U V) \\
	&+ \sum_{i} \{g((\nabla_{X_i}T)_U V,X_i)+g(A_{X_i}U, A_{X_i}V)\}, \nonumber\\
	(ii) \,\,\,\,\,\, S^M(X,Y) &=S^G (X', Y') \circ \pi + \frac{1}{2}
	\{g(\nabla_X N,Y)+g(\nabla_Y N, X) \} \label{S2}\\
	&-2 \sum_{i} g(A_X X_i, A_Y X_i) - \sum_{j} g(T_{U_j}X,T_{U_j}Y),\nonumber\\
	(iii) \,\,\,\,\,\, S^M(U,X) &= g(\nabla_U N,X) -\sum_{j} g(\nabla_{U_j}T)_{U_j}U,X) \label{S3} \\
	&+ \sum_{i} \{g((\nabla_{X_i}A)_{X_i} X,U)-2g(A_X {X_i}, T_U X_i)\}. \nonumber
	\end{align}
\end{proposition}

\begin{proposition}\label{pro2} \cite{FIP}
	Let's take the scalar curvatures of $(M, g), (G, g^\prime)$ Riemannian manifolds and $x \in G, \pi^{-1} (x)$ fibre with $r^M, r^G$ and $\hat{r}$, respectively.
	In a
	$$\pi: (M, g) \to (G, g^\prime)$$
	Riemannian submersion, $(M, g)$ depends on the scalar curvature of the Riemannian manifold $r^G$ and the scalar curve of any lift $\hat{r}$. In this case
	\begin{equation}\label{Pro2}
	r^M= \hat{r}+r^G \circ \pi - ||N||^2-||A||^2- ||T||^2+ 2 \sum_{i} g(\nabla_{X_i} N, X_i).
	\end{equation}
\end{proposition}

\section{Weyl projective curvature tensor along a Riemannian submersion}
In this section, we examine the Weyl projective curvature tensor relations between the total space, the base space and fibres on  a Riemannian submersion. We also give a corollary in case of the Riemannian submersion has totally umbilical fibres.
\begin{definition}\cite{mishra}
	Let take an $n$-dimensional differentiable manifold $M$ with
	differentiability class $C^{\infty }$. In the $n-$dimensional space $V_n$, the tensor
	\begin{equation*}
	P^*(X,Y)Z=R^M(X,Y)Z-\frac{1}{n-1}\{S^M(Y,Z)X-S^M(X,Z)Y\}. \label{5.3}
	\end{equation*}
	is called Weyl projective curvature tensor, where Ricci tensor of total space denoted by $S^{M}.$
\end{definition}
Now, we have the following main theorem.
\begin{theorem}
	Let $(M, g)$ and $(G, g^\prime)$ Riemannian manifolds,
	$$\pi: (M, g) \to(G, g^\prime)$$
	a Riemannian submersion and $R^M$, $R^G$ and $\hat{R}$ be Riemannian curvature tensors, $S^M$, $S^G$ and $\hat{S}$ be Ricci tensors of $M, \, G$ and the fibre respectively. Then for any $U, V, W, F \in \chi^v (M)$ and $X, Y, Z, H\in  \chi^h (M)$, we have the following relations for Weyl projective curvature tensor:
	\begin{align*}
		&g(P^*(X,Y)Z, H)\!\!=\!\!g(R^G (X,Y)Z, H)\!+\!2g (A_{X}Y, A_{Z}H)\!-\!g (A_{Y}Z, A_{X}H)\!+\!g (A_{X}Z, A_{Y}H)\\
		&- \frac{1}{n-1} \Bigg\{ g(X,H) \bigg[S^G (Y', Z') \circ \pi + \frac{1}{2} \left(g(\nabla_Y N,Z)+ g(\nabla_Z N,Y )  \right) \\
		&-2 \sum_{i} g(A_Y X_i, A_Z X_i) - \sum_{j} g(T_{U_j}Y, T_{U_j}Z)\bigg] \\
		&- g(Y,H) \bigg[S^G (X', Z') \circ \pi + \frac{1}{2} \left(g(\nabla_X N,Z)+ g(\nabla_Z N,X )  \right) \\
		&-2 \sum_{i} g(A_X X_i, A_Z X_i) - \sum_{j} g(T_{U_j}X, T_{U_j}Z)\bigg] \Bigg\},
	\end{align*}
	
	\begin{align*}
		g(P^*(X,Y)Z, V) &=-g((\nabla_Z A)_X Y, V) - g(A_X Y, T_V Z)+ g(A_Y Z, T_V X)\\
		&- g(A_X Z, T_V Y),
	\end{align*}
	
	\begin{align*}
		g(P^*(X,Y)V, W) &= g((\nabla_V A)_X Y, W) - g((\nabla_W A)_X Y, V) + g(A_X V, A_Y W)\\
		& - g(A_X W, A_Y V)- g(T_V X, T_W Y)+ g(T_W X, T_V Y),
	\end{align*}
	
	\begin{align*}
		g(P^*(X,V)Y, W) &=  g((\nabla_X T)_V W, Y) + g((\nabla_V A)_X Y, W) 
		- g(T_V X, T_W Y)\\
		& + g(A_X V, A_Y W),
	\end{align*}
	\begin{align*}
		g(P^*(U,V)W, X)&=g((\nabla_U T)_V W, X) - g((\nabla_V T)_U W, X)
	\end{align*}
	and
	\begin{align*}
		g(P^*(U,V)W, F)&= g( \hat{R} (U,V)W , F)+ g (T_{U}W, T_{V}F )- g (T_{V}W, T_{U}F) \\
		&- \frac{1}{n-1} \Bigg\{ g(F,U) \bigg[\hat{S} (V, W) - g(N,T_V W)  \\
		&+  \sum_{i} \left( g((\nabla_{X_i} T)_V W, X_i)+g(A_{X_i} V, A_{X_i} W)\right)\bigg] \\
		&- g(F,V) \bigg[\hat{S} (U, W) - g(N,T_U W)   \\
		&+  \sum_{i} \left( g((\nabla_{X_i} T)_U W, X_i)+g(A_{X_i} U, A_{X_i} W)\right)\bigg]\Bigg\}.
	\end{align*}
\end{theorem}
\begin{proof}
	We only give the proof of the $1^{st}$ equation of this theorem. The following equations are obtained inner production with $H$ to $P^*$ and using \eqref{g2} and \eqref{S2} equations.
	\begin{align}
		g(R^M(X,Y)Z,H)&=g(R^G(X,Y)Z,H)+ 2g(A_X Y, A_Z H)- g(A_Y Z,A_X H)\nonumber \\
		&+g(A_X Z, A_Y H),\nonumber 
	\end{align}
	
	\begin{align}
		S^M(Y,Z) &= S^G (Y',Z')  \circ \pi	+ \frac{1}{2} \left\{g(\nabla_Y N,Z)+g(\nabla_Z N, Y)\right\} \nonumber \\
		&-2\sum_{i} g(A_Y Y_i, A_Z Y_i) -\sum_{j}(T_{U_j}Y, T_{U_j}Z)\nonumber 
	\end{align}
	and
	\begin{align}
		S^M(X,Z) &= S^G (Y', Z')  \circ \pi
		+ \frac{1}{2} \left\{g(\nabla_X N,Z)+g(\nabla_Z N, X)\right\} \nonumber\\
		&-2\sum_{i} g(A_X X_i, A_Z X_i) -\sum_{j}(T_{U_j}X, T_{U_j}Z).\nonumber 
	\end{align}
	When these equations are substituted in $P^*$, the given result is obtained. Other equations are similarly proved by using Theorem \ref{thm1} and Proposition \ref{pro1}.
\end{proof}

\begin{corollary}
Let $\pi: (M, g) \to(G, g^\prime)$ be a Riemannian submersion, where $(M, g)$ and $(G, g^\prime)$ Riemannian manifolds. If the Riemannian submersion has total umbilical fibres, that  is $N = 0$,  then the Weyl projective curvature tensor is given by
	\begin{align*}
		&g(P^*(X,Y)Z, H)\\
		&=g(R^G (X,Y) Z , H )+2g (A_{X}Y, A_{Z}H)-g (A_{Y}Z, A_{X}H )+g (A_{X}Z, A_{Y}H )\\
		&- \frac{1}{n-1} \Bigg\{ g(X,H) \bigg[S^G (Y', Z') \circ \pi
		-2 \sum_{i} g(A_Y X_i, A_Z X_i) - \sum_{j} g(T_{U_j}Y, T_{U_j}Z)\bigg] \\
		&- g(Y,H) \bigg[S^G (X', Z') \circ \pi
		-2 \sum_{i} g(A_X X_i, A_Z X_i) - \sum_{j} g(T_{U_j}X, T_{U_j}Z)\bigg] \Bigg\},
	\end{align*}
	and
	\begin{align*}
		&g(P^*(U,V)W, F) \\
		&= g( \hat{R} (U,V)W , F)+ g (T_{U}W, T_{V}F )- g (T_{V}W, T_{U}F) \\
		&- \frac{1}{n-1} \Bigg\{ g(F,U) \bigg[\hat{S} (V, W)
		+  \sum_{i} \left( g((\nabla_{X_i} T)_V W, X_i)+g(A_{X_i} V, A_{X_i} W)\right)\bigg] \\
		&- g(F,V) \bigg[\hat{S} (U, W)
		+  \sum_{i} \left( g((\nabla_{X_i} T)_U W, X_i)+g(A_{X_i} U, A_{X_i} W)\right)\bigg]\Bigg\}.
	\end{align*}
for any $U, V, W, F \in \chi^v (M)$ and $X, Y, Z, H\in  \chi^h (M)$, 
\end{corollary}
\section{Concircular curvature tensor along a Riemannian submersion}
In this section, curvature relations of concircular curvature tensor in a Riemannian submersion are examined. In particular, we show  that the Riemannian submersion with concircular curvature tensor has no the totally umbilical fibres.
\begin{definition}\rm \label{defconcir}
	In the $n-$dimensional space $V_n$, the tensor
	\begin{align*}
		C^*(X,Y,Z,H)&=R^M(X,Y,Z,H)- \frac{r^M}{n(n-1)} [g(X,H)g(Y,Z)-g(Y,H)g(X,Z)],
	\end{align*}
	is called concircular curvature tensor, where scalar tensor denoted by $r^M$\cite{mishra}.
\end{definition}
Now, we have the following main theorem.
\begin{theorem}
	Let $(M, g)$ and $(G, g^\prime)$ Riemannian manifolds,
	$$\pi: (M, g) \to(G, g^\prime)$$
	a Riemannian submersion and $R^M$, $R^G$ and  $\hat{R}$ be Riemannian curvature tensors, $r^M$, $r^G$ and  $\hat{r}$ be scalar curvature tensors of $M, \, G$ and the fibre respectively. Then for any $U, V, W, F \in \chi^v (M)$ and $X, Y, Z, H\in  \chi^h (M)$, we have the following relations
	\begin{align*}
		g(C^*(X,Y)Z, H) &=  g(R^G (X,Y) Z , H )+ 2g (A_{X}Y, A_{Z}H )\\
		&- g (A_{Y}Z, A_{X}H )+ g (A_{X}Z, A_{Y}H )\\
		&- \frac{r^M}{n(n-1)} \Bigg\{ g(Y,Z) g(X,H)- g(X,Z) g(Y,H) \Bigg\},
	\end{align*}
	
	\begin{align*}
		g(C^*(X,Y)Z, V) &= -g((\nabla_Z A)_X Y, V) - g(A_X T, T_V Z) \\
		&+ g(A_Y Z, T_V X) - g(A_X Z, T_V Y),
	\end{align*}
	
	\begin{align*}
		g(C^*(X,Y)V, W) &= g((\nabla_V A)_X Y, W) - g((\nabla_W A)_X Y, V) + g(A_X V, A_Y W) \\
		&- g(A_X W, A_Y V)-g(T_V X, T_W Y)+ g(T_W X, T_V Y),
	\end{align*}
	
	\begin{align*}
		g(C^*(X,V)Y, W) &= g((\nabla_X T)_V W, Y) + g((\nabla_V A)_X Y, W)
		- g(T_V X, T_W Y) \\
		&+ g(A_X Y, A_Y W)-\frac{r^M}{n(n-1)} \{-g(X,Y)g(V,W)\},
	\end{align*}
	\begin{align*}
		g(C^*(U,V)W, X) &= g((\nabla_U T)_V W, X) - g((\nabla_V T)_U W, X)
	\end{align*}
	and
	\begin{align*}
		g(C^*(U,V)W, F)&= g( \hat{R} (U,V)W , F)+ g (T_{U}W, T_{V}F )- g (T_{V}W, T_{U}F) \\
		&- \frac{r^M}{n(n-1)} \Bigg\{ g(V,W)g(U,F)- g(U,W)g(V,F)\Bigg\}
	\end{align*}
	where
	\begin{equation*}
	r^M= \hat{r}+r^G \circ \pi -||A||^2- ||T||^2.
	\end{equation*}
\end{theorem}
\begin{proof}
	Let's prove the $2^{nd}$ equation of this theorem. Taking inner product $C^*$ with $V$ then we have
	\begin{align*}
		g(C^*(X,Y)Z,V) = g(R(X,Y)Z,V) -\frac{r^M}{n(n-1)} \{g(Y,Z)g(X,V)-g(X,Z)\}.
	\end{align*}
	Then using equation \eqref{g3}, we get
	\begin{align*}
		g(C^*(X,Y)Z,V)&=-g((\nabla_Z A)_X Y,V)- g(A_X T, T_V Z) \\
		&+g(A_Y Z, T_V X)- g(A_X Z, T_V Y).
	\end{align*}
which completes the proof of the second equation. Other equations are similarly proved by using Theorem \ref{thm1}, Proposition \ref{pro1} and Proposition \ref{pro2}.
\end{proof}
\begin{corollary}
Let $\pi: (M, g) \to(G, g^\prime)$ be a Riemannian submersion, where $(M, g)$ and $(G, g^\prime)$ Riemannian manifolds. Then
the concircular curvature tensor of Riemannian submersion has no total umbilical fibres.
\end{corollary}

\section{Conharmonic curvature tensor along a Riemannian submersion}
In this section, curvature relations of conharmonic curvature tensor in a Riemannian submersion are examined.
\begin{definition}\rm \label{defconhar}
	In the $n-$dimensional space $V_n$, the tensor
	\begin{align*}
		L^*(X,Y,Z,H)&= R^M(X,Y,Z,H)- \frac{1}{n-2}[g(Y,Z)Ric(X,H)-g(X,Z)Ric(Y,H),
	\end{align*}
	is called conharmonic curvature tensor, where Ricci tensor denoted by $Ric$ \cite{mishra}.
\end{definition}
In a similar way, we have the following main theorem.
\begin{theorem}
	Let $(M, g)$ and $(G, g^\prime)$ Riemannian manifolds,
	$$\pi: (M, g) \to(G, g^\prime)$$
	a Riemannian submersion and $R^M$, $R^G$ and $\hat{R}$ be Riemannian curvature tensors, $S^M$, $S^G$ and $\hat{S}$ be Ricci tensors of $M, \, G$ and the fibre respectively. Then for any $U, V, W, F \in \chi^v (M)$ and $X, Y, Z, H\in  \chi^h (M)$, we have the following relations
	\begin{align*}
		&g(L^* (X,Y)Z, H)\\
		&= g(R^G (X,Y)Z, H) +2 g (A_X Y, A_Z H) - g(A_Y Z, A_X H) + g(A_X Z, A_Y H) \\
		&- \frac{1}{(n-2)} \Bigg\{g(Y,Z) \bigg[S^G (X',H') \circ \pi + \frac{1}{2} \left(g(\nabla_X N,H)+ g(\nabla_H N, X)\right)  \\
		&-2 \sum_{i} g(A_X X_i, A_H X_i)- \sum_{j}g(T_{U_j}X, T_{U_j}H)  \bigg] \\
		&-g(X,Z) \bigg[S^G (Y',H') \circ \pi + \frac{1}{2} \left(g(\nabla_Y N,H)+ g(\nabla_H N, Y)\right) \\
		&-2 \sum_{i} g(A_Y X_i, A_H X_i)- \sum_{j}g(T_{U_j}Y, T_{U_j}H)  \bigg] \\
		&+g(X, H) \bigg[S^G (Y',Z') \circ \pi
		+\frac{1}{2} \left(g(\nabla_Y N,Z)+g(\nabla_Z N, Y)\right) \\
		&- 2 \sum_{i} g(A_Y X_i, A_Z X_i) - \sum_{j} g(T_{U_j}Y, T_{U_j}Z) \Bigg] \\
		&-g(Y,H) \bigg[S^G (X',Z') \circ \pi + \frac{1}{2} \left(g(\nabla_X N,Z)+ g(\nabla_Z N, X)\right) \\
		&-2 \sum_{i} g(A_X X_i, A_Z X_i)- \sum_{j}g(T_{U_j}X, T_{U_j}Z)  \bigg] \Bigg\},
	\end{align*}
	
	\begin{align*}
		&g(L^* (X,Y)Z, V) \\
		&=-g((\nabla_Z A)_X Y,V) - g(A_X T, T_V Z) + g(A_Y Z, T_V X)- g(A_X Z, T_V Y) \\
		&- \frac{1}{(n-2)} \Bigg\{g(Y, Z) \bigg[g(\nabla_X N,V)
		-\sum_{j} g((\nabla_{U_j} T)_{U_j} X, V) \\
		&+\sum_{i}\left(g((\nabla_{X_i}A)_{X_i} X,V )- 2 g (A_V X_i, T_X X_i) \right) \bigg] \\
		&  -g(X,Z) \bigg[g(\nabla_Y N,V)
		-\sum_{j} g((\nabla_{U_j} T)_{U_j} Y, V) \\
		&+ \sum_{i}\left(g((\nabla_{X_i}A)_{X_i} Y,V  ) - 2 g (A_V X_i, T_Y X_i) \right) \bigg] \Bigg\},
	\end{align*}
	
	\begin{align*}
		g(L^* (X,Y)V, W)&= g((\nabla_V A)_X Y, W)- g((\nabla_W A)_X Y, V)+ g(A_X V, A_Y W)\\
		&- g(A_X W, A_Y V) - g(T_V X, T_W Y) + g(T_W X, T_V Y),
	\end{align*}
	
	\begin{align*}
		&g(L^* (X,V)Y ,W)\\
		&=g((\nabla_X T)_V W, Y) + g((\nabla_V A)_X Y,W)- g(T_V X, T_W Y) + g(A_X Y, A_Y W)\\
		&-\frac{1}{(n-2)} \bigg\{ - g(V,W) \Big[ S^G (X', Y') \circ \pi + \frac{1}{2} \left(g(\nabla_X N,Y)+ g(\nabla_Y N,X)\right) \\
		&- 2 \sum_{i} g(A_X X_i, A_Y X_i) -\sum_{j} g(T_{U_j}X, T_{U_j}Y) \Big]\\
		& - g(X,Y) \Big[\hat{S} (V,W) -g(N,T_V W) + \sum_{i}\left(g((\nabla_{X_i}T)_V W,X_i) \right.\\
		& \left. + g(A_{X_i} V, A_{X_i}W)  \right) \Big] \bigg\},
	\end{align*}
	\begin{align*}
		g(L^* (U,V)W ,X)
		&=g((\nabla_U T)_V W, X) - g((\nabla_V T)_U W, X)\\
		&- \frac{1}{(n-2)} \bigg\{g(V,W) \Big[g(\nabla_U N,X) -\sum_{j} g(\nabla_{U_j}T)_{U_j}U,X)  \\
		&+ \sum_{i} \{g((\nabla_{X_i}A)_{X_i} X,U)-2g(A_X {X_i}, T_U X_i)\}\Big] \\
		&- g(U,W)  \Big[g(\nabla_V N,X) -\sum_{j} g(\nabla_{U_j}T)_{U_j}V,X)  \\
		&+ \sum_{i} \{g((\nabla_{X_i}A)_{X_i} X,V)-2g(A_X {X_i}, T_V X_i)\}\Big] \bigg\}
	\end{align*}
	and
	\begin{align*}
		g(L^* (U,V)W ,F)&= g(\hat{R} (U,V)W, F) + g(T_U W, T_V F) - g(T_V W, T_U F ) \\
		&- \frac{1}{(n-2)} \bigg\{g(V,W) \Big[\hat{S}(U,F) -g(N, T_U F)\\
		&+ \sum_{i}\left(g((\nabla_{X_i}T)_U F, X_i)+g(A_{X_i}U,A_{X_i}F)\right)\Big]\\
		&-g(U,W) \Big[\hat{S}(V,F)-g(N,T_V F) \\
		&+ \sum_{i}\left(g((\nabla_{X_i}T)_V F, X_i)+ g(A_{X_i} V, A_{X_i}F)\right)\Big] \\
		&+g(F,U) \Big[ \hat{S}(U,V)- g(N, T_U V) \\
		& + \sum_{i}\left(g((\nabla_{X_i}T)_U V, {X_i})+ g(A_{X_i}U, A_{X_i}V)\right)  \Big] \\
		&-g(F,V) \Big[ \hat{S}(U,W) -g(N,T_U W)\\
		&+ \sum_{i}\left( g((\nabla_{X_i}T)_U W, X_i)+ g(A_{X_i}U, A_{X_i}W) \right) \Big] \bigg\}.
	\end{align*}
\end{theorem}
\begin{proof}
	Let's prove the $3^{th}$ equation of this theorem. The following equations are obtained inner production with $W$ to $L^*$ and by using equation \eqref{g4}
	\begin{align*}
		g(L^*(X,Y)V,W)&=g(R^M(X,Y)V,W)-\frac{1}{n-2} \{g(X,W)S(Y,V) \\
		&-g(Y,W)S(X,V)+g(Y,V)S(X,W)-g(X,V)S(Y,W)\}. \nonumber
	\end{align*}
	One can easily obtain the other equations by using Theorem \ref{thm1} and Proposition \ref{pro1}.
\end{proof}
\begin{corollary}
Let $\pi: (M, g) \to(G, g^\prime)$ be a Riemannian submersion, where $(M, g)$ and $(G, g^\prime)$ Riemannian manifolds. If the Riemannian submersion has total umbilical fibres, that  is $N = 0$,  then the conharmonic curvature tensor is given by
	\begin{align*}
		&g(L^* (X,Y)Z, H)\\
		&= g(R^G (X,Y)Z, H) +2 g (A_X Y, A_Z H) - g(A_Y Z, A_X H) + g(A_X Z, A_Y H) \\
		&- \frac{1}{(n-2)} \Bigg\{g(Y,Z) \bigg[S^G (X',H') \circ \pi
		-2 \sum_{i} g(A_X X_i, A_H X_i)- \sum_{j}g(T_{U_j}X, T_{U_j}H)  \bigg] \\
		&-g(X,Z) \bigg[S^G (Y',H') \circ \pi
		-2 \sum_{i} g(A_Y X_i, A_H X_i)- \sum_{j}g(T_{U_j}Y, T_{U_j}H)  \bigg] \\
		&+g(X, H) \bigg[S^G (Y',Z') \circ \pi
		- 2 \sum_{i} g(A_Y X_i, A_Z X_i) - \sum_{j} g(T_{U_j}Y, T_{U_j}Z) \Bigg] \\
		&-g(Y,H) \bigg[S^G (X',Z') \circ \pi
		-2 \sum_{i} g(A_X X_i, A_Z X_i)- \sum_{j}g(T_{U_j}X, T_{U_j}Z)  \bigg] \Bigg\},
	\end{align*}
	
	\begin{align*}
		g(L^* (X,Y)Z, V) &=-g((\nabla_Z A)_X Y,V)\!-\!g(A_X T, T_V Z) \!+\! g(A_Y Z, T_V X)\!- \!g(A_X Z, T_V Y) \\
		&- \frac{1}{(n-2)} \Bigg\{g(Y, Z) \bigg[
		-\sum_{j} g((\nabla_{U_j} T)_{U_j} X, V) \\
		&+\sum_{i}\left(g((\nabla_{X_i}A)_{X_i} X,V )- 2 g (A_V X_i, T_X X_i) \right) \bigg] \\
		&  -g(X,Z) \bigg[-\sum_{j} g((\nabla_{U_j} T)_{U_j} Y, V) \\
		&+ \sum_{i}\left(g((\nabla_{X_i}A)_{X_i} Y,V  ) - 2 g (A_V X_i, T_Y X_i) \right) \bigg] \Bigg\},
	\end{align*}
	
	\begin{align*}
		&g(L^* (X,V)Y ,W)\\
		&= g((\nabla_X T)_V W, Y) + g((\nabla_V A)_X Y,W)- g(T_V X, T_W Y) + g(A_X Y, A_Y W)\\
		&-\frac{1}{(n-2)} \bigg\{ - g(V,W) \Big[ S^G (X', Y') \circ \pi
		- 2 \sum_{i} g(A_X X_i, A_Y X_i) -\sum_{j} g(T_{U_j}X, T_{U_j}Y) \Big]\\
		& - g(X,Y) \Big[\hat{S} (V,W) + \sum_{i}\left(g((\nabla_{X_i}T)_V W,X_i)
		+ g(A_{X_i} V, A_{X_i}W)  \right) \Big] \bigg\},
	\end{align*}
	\begin{align*}
		g(L^* (U,V)W ,X)&= g((\nabla_U T)_V W, X) - g((\nabla_V T)_U W, X)\\
		&- \frac{1}{(n-2)} \bigg\{g(V,W) \Big[\sum_{j} g(\nabla_{U_j}T)_{U_j}U,X)  \\
		&+ \sum_{i} \{g((\nabla_{X_i}A)_{X_i} X,U)-2g(A_X {X_i}, T_U X_i)\}\Big] \\
		&- g(U,W)  \Big[\sum_{j} g(\nabla_{U_j}T)_{U_j}V,X)  \\
		&+ \sum_{i} \{g((\nabla_{X_i}A)_{X_i} X,V)-2g(A_X {X_i}, T_V X_i)\}\Big] \bigg\}
	\end{align*}
	and
	\begin{align*}
		&g(L^* (U,V)W ,F)\\
		&= g(\hat{R} (U,V)W, F) + g(T_U W, T_V F) - g(T_V W, T_U F ) \\
		&- \frac{1}{(n-2)} \bigg\{g(V,W) \Big[\hat{S}(U,F) \\
		&+ \sum_{i}\left(g((\nabla_{X_i}T)_U F, X_i)+g(A_{X_i}U,A_{X_i}F)\right)\Big]\\
		&-g(U,W) \Big[\hat{S}(V,F)
		+ \sum_{i}\left(g((\nabla_{X_i}T)_V F, X_i)+ g(A_{X_i} V, A_{X_i}F)\right)\Big] \\
		&+g(F,U) \Big[ \hat{S}(U,V)
		+ \sum_{i}\left(g((\nabla_{X_i}T)_U V, {X_i})+ g(A_{X_i}U, A_{X_i}V)\right)  \Big] \\
		&-g(F,V) \Big[ \hat{S}(U,W)
		+ \sum_{i}\left( g((\nabla_{X_i}T)_U W, X_i)+ g(A_{X_i}U, A_{X_i}W) \right) \Big] \bigg\}.
	\end{align*}
\end{corollary}

\section{Conformal curvature tensor along a Riemannian submersion}
In this section, we find some curvature relations of conformal curvature tensor in a Riemannian submersion and give a corollary in case of the Riemannian submersion has totally umbilical fibres.
\begin{definition}\rm \label{def1}
	In the $n-$dimensional space $V_n$, the tensor
	\begin{align*}
		V^*(X,Y,Z,H)&=R^M(X,Y,Z,H)-\frac{1}{n-2} [g(X,H)Ric(Y,Z)-g(Y,H)Ric(X,Z)\\
		&+g(Y,Z)Ric(X,H)-g(X,Z)Ric(Y,H)]\\
		&+\frac{r^M}{(n-1)(n-2)} [g(X,H)g(Y,Z)-g(Y,H)g(X,Z)],
	\end{align*}
	is called conformal curvature tensor, where Ricci tensor and scalar tensor denoted by $Ric$ and $r^M$  respectively \cite{mishra}.
\end{definition}
\begin{theorem}
	Let $(M, g)$ and $(G, g^\prime)$ Riemannian manifolds,
	$$\pi: (M, g) \to(G, g^\prime)$$
	a Riemannian submersion and $R^M$, $R^G$ and $\hat{R}$ be Riemannian curvature tensors, $S^M$, $S^G$ and $\hat{S}$ be Ricci tensors and $r^M$, $r^G$ and $\hat{r}$ be scalar curvature tensors of $M, \, G$ and the fibre respectively. Then for any $U, V, W, F \in \chi^v (M)$ and $X, Y, Z, H\in  \chi^h (M)$, we have the following relations
	\begin{align*}
		&g(V^* (X,Y)Z, H) \\
		&= g(R^G (X,Y)Z, H) +2 g (A_X Y, A_Z H) - g(A_Y Z, A_X H) + g(A_X Z, A_Y H) \\
		&- \frac{1}{(n-2)} \Bigg\{g(X, H) \bigg[S^G (Y',Z') \circ \pi
		+\frac{1}{2} \left(g(\nabla_Y N,Z)+g(\nabla_Z N, Y)\right) \\
		&- 2 \sum_{i} g(A_Y X_i, A_Z X_i) - \sum_{j} g(T_{U_j}Y, T_{U_j}Z) \Bigg] \\
		&-g(Y,H) \bigg[S^G (X',Z') \circ \pi + \frac{1}{2} \left(g(\nabla_X N,Z)+ g(\nabla_Z N, X)\right) \\
		&-2 \sum_{i} g(A_X X_i, A_Z X_i)- \sum_{j}g(T_{U_j}X, T_{U_j}Z)  \bigg] \\
		&+g(Y,Z) \bigg[S^G (X',H') \circ \pi + \frac{1}{2} \left(g(\nabla_X N,H)+ g(\nabla_H N, X)\right)  \\
		&-2 \sum_{i} g(A_X X_i, A_H X_i)- \sum_{j}g(T_{U_j}X, T_{U_j}H)  \bigg] \\
		&-g(X,Z) \bigg[S^G (Y',H') \circ \pi + \frac{1}{2} \left(g(\nabla_Y N,H)+ g(\nabla_H N, Y)\right) \\
		&-2 \sum_{i} g(A_Y X_i, A_H X_i)- \sum_{j}g(T_{U_j}Y, T_{U_j}H)  \bigg]\Bigg\} \\
		&+ \frac{r^M}{(n-1)(n-2)} \{g(Y,Z) g(X,H)- g(X,Z)g(Y,H)\},
	\end{align*}
	\begin{align*}
		g(V^* (X,Y)Z, V) &\!=\! -g((\nabla_Z A)_X Y,V)\!\!-\!\! g(A_X T, T_V Z)\!+\!g(A_Y Z, T_V X)\!-\!g(A_X Z, T_V Y) \\
		&- \frac{1}{(n-2)} \Bigg\{g(Y, Z) \bigg[g(\nabla_X N,V)
		-\sum_{j} g((\nabla_{U_j} T)_{U_j} X, V) \\
		&+\sum_{i}\left(g((\nabla_{X_i}A)_{X_i} X,V )- 2 g (A_V X_i, T_X X_i) \right) \bigg] \\
		&  -g(X,Z) \bigg[g(\nabla_Y N,V)
		-\sum_{j} g((\nabla_{U_j} T)_{U_j} Y, V) \\
		&+ \sum_{i}\left(g((\nabla_{X_i}A)_{X_i} Y,V  ) - 2 g (A_V X_i, T_Y X_i) \right) \bigg] \Bigg\},
	\end{align*}
	\begin{align*}
		g(V^* (X,Y)V, W)&=& g((\nabla_V A)_X Y, W)- g((\nabla_W A)_X Y, V)+ g(A_X V, A_Y W)\\
		&& - g(A_X W, A_Y V) - g(T_V X, T_W Y) + g(T_W X, T_V Y),
	\end{align*}
	\begin{align*}
		&g(V^* (X,V)Y ,W)\\
		&=g((\nabla_X T)_V W, Y) + g((\nabla_V A)_X Y,W)- g(T_V X, T_W Y) + g(A_X Y, A_Y W)\\
		&-\frac{1}{(n-2)} \bigg\{ - g(V,W) \Big[ S^G (X', Y') \circ \pi + \frac{1}{2} \left(g(\nabla_X N,Y)+ g(\nabla_Y N,X)\right) \\
		&- 2 \sum_{i} g(A_X X_i, A_Y X_i) -\sum_{j} g(T_{U_j}X, T_{U_j}Y) \Big]\\
		& - g(X,Y) \Big[\hat{S} (V,W) -g(N,T_V W) + \sum_{i}\left(g((\nabla_{X_i}T)_V W,X_i) \right.\\
		& \left. + g(A_{X_i} V, A_{X_i}W)  \right) \Big] \bigg\},
	\end{align*}
	\begin{align*}
		g(V^* (U,V)W ,X)&= g((\nabla_U T)_V W, X) - g((\nabla_V T)_U W, X)\\
		&- \frac{1}{(n-2)} \bigg\{g(V,W) \Big[g(\nabla_U N,X) -\sum_{j} g(\nabla_{U_j}T)_{U_j}U,X)  \\
		&+ \sum_{i} \{g((\nabla_{X_i}A)_{X_i} X,U)-2g(A_X {X_i}, T_U X_i)\}\Big] \\
		&- g(U,W)  \Big[g(\nabla_V N,X) -\sum_{j} g(\nabla_{U_j}T)_{U_j}V,X)  \\
		&+ \sum_{i} \{g((\nabla_{X_i}A)_{X_i} X,V)-2g(A_X {X_i}, T_V X_i)\}\Big] \bigg\}
	\end{align*}
	and
	\begin{align*}
		&g(V^* (U,V)W ,F)=g(\hat{R} (U,V)W, F) + g(T_U W, T_V F) - g(T_V W, T_U F ) \\
		&- \frac{1}{(n-2)} \bigg\{ g(F,U) \Big[ \hat{S}(V,W)- g(N, T_V W) \\
		& + \sum_{i}\left(g(\nabla_{X_i}T)_V W, {X_i}+ g(A_{X_i}V, A_{X_i}W)\right)  \Big] \\
		&-g(F,V) \Big[ \hat{S}(U,W) -g(N,T_U W)+ \sum_{i} \left(g((\nabla_{X_i}T)_U W, X_i)+ g(A_{X_i}U, A_{X_i}W) \right)\Big]\\
		&+g(V,W) \Big[\hat{S}(U,F) -g(N, T_U F)+ \sum_{i}\left(g((\nabla_{X_i}T)_U F, X_i)+g(A_{X_i}U,A_{X_i}F)\right)\Big]\\
		&-g(U,W) \Big[\hat{S}(V,F)-g(N,T_V F) + \sum_{i}\left(g((\nabla_{X_i}T)_V F, X_i)+ g(A_{X_i} V, A_{X_i}F)\right)\Big] \bigg\}\\
		&+\frac{r^M}{(n-1)(n-2)} \{g(V,W)g(U,F)-g(U,W)g(V,F)\}
	\end{align*}
	where
	\begin{equation*}
	r^M= \hat{r}+r^G \circ \pi - ||N||^2-||A||^2- ||T||^2+ 2 \sum_{i} g(\nabla_{X_i} N, X_i).
	\end{equation*}
\end{theorem}
\begin{proof}
	Let's prove the $4^{th}$ equation of this theorem. The following equations are obtained inner production with $W$ to $V^*$
	\begin{align*}
		g(V^*(X,V)Y,W)&=g(R^M(X,V)Y,W) -\frac{1}{n-2} \{g(X,W)S^M(V,Y) \\
		&-g(V,W)S^M(X,Y)+g(V,Y)S(X,W)-g(X,Y)S^M(V,W)\} \nonumber\\
		&+ \frac{r^M}{(n-1)(n-2)} \{g(X,W)g(Y,V)-g(X,V)g(Y,W)\}. \nonumber
	\end{align*}
	Then using equations \eqref{g5}-\eqref{S2}, we have the desired result.
	From the Theorem \ref{thm1}, Proposition \ref{pro1} and  Proposition \ref{pro2} the above equations are obtained.
\end{proof}

\begin{corollary}
Let $\pi: (M, g) \to(G, g^\prime)$ be a Riemannian submersion, where $(M, g)$ and $(G, g^\prime)$ Riemannian manifolds. If the Riemannian submersion has total umbilical fibres, that  is $N = 0$,  then the conformal curvature tensor is given by
	\begin{align*}
		&g(V^* (X,Y)Z, H) \\
		&=g(R^G (X,Y)Z, H) +2 g (A_X Y, A_Z H) - g(A_Y Z, A_X H) + g(A_X Z, A_Y H) \\
		&- \frac{1}{(n-2)} \Bigg\{g(X, H) \bigg[S^G (Y',Z') \circ \pi
		- 2 \sum_{i} g(A_Y X_i, A_Z X_i) - \sum_{j} g(T_{U_j}Y, T_{U_j}Z) \Bigg] \\
		&-g(Y,H) \bigg[S^G (X',Z') \circ \pi
		-2 \sum_{i} g(A_X X_i, A_Z X_i)- \sum_{j}g(T_{U_j}X, T_{U_j}Z)  \bigg] \\
		&+g(Y,Z) \bigg[S^G (X',H') \circ \pi
		-2 \sum_{i} g(A_X X_i, A_H X_i)- \sum_{j}g(T_{U_j}X, T_{U_j}H)  \bigg] \\
		&-g(X,Z) \bigg[S^G (Y',H') \circ \pi
		-2 \sum_{i} g(A_Y X_i, A_H X_i)- \sum_{j}g(T_{U_j}Y, T_{U_j}H)  \bigg]\Bigg\} \\
		&+ \frac{r^M}{(n-1)(n-2)} \{g(Y,Z) g(X,H)- g(X,Z)g(Y,H)\},
	\end{align*}
	\begin{align*}
		&g(V^* (X,Y)Z,V) \\
		&= -g((\nabla_Z A)_X Y,V) - g(A_X T, T_V Z) + g(A_Y Z, T_V X)- g(A_X Z, T_V Y) \\
		&- \frac{1}{(n-2)} \Bigg\{g(Y, Z) \bigg[-\sum_{j} g((\nabla_{U_j} T)_{U_j} X, V)
		+\sum_{i}\left(g((\nabla_{X_i}A)_{X_i} X,V )- 2 g (A_V X_i, T_X X_i) \right) \bigg] \\
		&-g(X,Z) \bigg[-\sum_{j} g((\nabla_{U_j} T)_{U_j} Y, V)
		+ \sum_{i}\left(g((\nabla_{X_i}A)_{X_i} Y,V  ) - 2 g (A_V X_i, T_Y X_i) \right) \bigg] \Bigg\},
	\end{align*}
	\begin{align*}
		&g(V^* (X,V)Y ,W)\\
		&=g((\nabla_X T)_V W, Y) + g((\nabla_V A)_X Y,W)- g(T_V X, T_W Y) + g(A_X Y, A_Y W)\\
		&-\frac{1}{(n-2)} \bigg\{ - g(V,W) \Big[ S^G (X', Y') \circ \pi
		- 2 \sum_{i} g(A_X X_i, A_Y X_i) -\sum_{j} g(T_{U_j}X, T_{U_j}Y) \Big]\\
		&- g(X,Y) \Big[\hat{S} (V,W) + \sum_{i}\left(g((\nabla_{X_i}T)_V W,X_i)
		+ g(A_{X_i} V, A_{X_i}W)  \right) \Big] \bigg\},
	\end{align*}
	\begin{align*}
		g(V^* (U,V)W ,X)&= g((\nabla_U T)_V W, X) - g((\nabla_V T)_U W, X)\\
		&- \frac{1}{(n-2)} \bigg\{g(V,W) \Big[\sum_{j} g(\nabla_{U_j}T)_{U_j}U,X)  \\
		&+ \sum_{i} \{g((\nabla_{X_i}A)_{X_i} X,U)-2g(A_X {X_i}, T_U X_i)\}\Big] \\
		&- g(U,W)  \Big[\sum_{j} g(\nabla_{U_j}T)_{U_j}V,X)  \\
		&+ \sum_{i} \{g((\nabla_{X_i}A)_{X_i} X,V)-2g(A_X {X_i}, T_V X_i)\}\Big] \bigg\}
	\end{align*}
	and
	\begin{align*}
		g(V^* (U,V)W ,F)&=g(\hat{R} (U,V)W, F) + g(T_U W, T_V F) - g(T_V W, T_U F ) \\
		&- \frac{1}{(n-2)} \bigg\{ g(F,U) \Big[ \hat{S}(V,W) \\
		&+ \sum_{i}\left(g(\nabla_{X_i}T)_V W, {X_i}+ g(A_{X_i}V, A_{X_i}W)\right)  \Big] \\
		&-g(F,V) \Big[ \hat{S}(U,W) + \sum_{i} \left(g((\nabla_{X_i}T)_U W, X_i)+ g(A_{X_i}U, A_{X_i}W) \right)\Big]\\
		&+g(V,W) \Big[\hat{S}(U,F) + \sum_{i}\left(g((\nabla_{X_i}T)_U F, X_i)+g(A_{X_i}U,A_{X_i}F)\right)\Big]\\
		&-g(U,W) \Big[\hat{S}(V,F) + \sum_{i}\left(g((\nabla_{X_i}T)_V F, X_i)+ g(A_{X_i} V, A_{X_i}F)\right)\Big] \bigg\}\\
		&+\frac{r^M}{(n-1)(n-2)} \{g(V,W)g(U,F)-g(U,W)g(V,F)\}
	\end{align*}
	where
	\begin{equation*}
	r^M= \hat{r}+r^G \circ \pi - ||A||^2- ||T||^2.
	\end{equation*}
\end{corollary}

Finally, we investigate the $M-$projective curvature tensor on a Riemannian submersion and give a corollary in case of the totally umbilical fibres.
\section{$M$-projective curvature tensor along a Riemannian submersion}
In this section, curvature relations of $M$-projective curvature tensor in a Riemannian submersion are examined and obtain a corollary using the curvature tensor.
\begin{definition}
	Let take an $n$-dimensional differentiable manifold $M^n$ with
	differentiability class $C^{\infty }$. In 1971 on a $n$%
	-dimensional Riemannian manifold, ones \cite{PM} defined a tensor
	field $W^{\ast }$ as
	\begin{align*}
	W^{\ast }(X,Y)Z&=R^M(X,Y)Z-\frac{1}{2(n-1)}[S^M(Y,Z)X\\
	&-S^M(X,Z)Y+g(Y,Z)QX-g(X,Z)QY]\label{5.0}
	\end{align*}
	tensor $W^{\ast }$ as $M$-projective curvature tensor.
\end{definition}
In addition, on an $n-$dimensional Riemannian manifold $M^n$ the Ricci operator $Q$ is defined by
\begin{align*}
S^M(X,Y) = g(QX,Y).
\end{align*}
\begin{theorem}
	Let $(M, g)$ and $(G, g^\prime)$ Riemannian manifolds,
	$$\pi: (M, g) \to(G, g^\prime)$$
	a Riemannian submersion and $R^M$, $R^G$ and  $\hat{R}$ be Riemannian curvature tensors, $S^M$, $S^G$ and  $\hat{S}$ be Ricci tensors of $M, \, G$ and the fibre respectively. Then for any $U, V, W, F \in \chi^v (M)$ and $X, Y, Z, H\in  \chi^h (M)$, we have the following relations for $M$-projective curvature tensor:
	\begin{align*}
		&g(W^* (X,Y)Z, H) \\
		&=g(R^G (X,Y)Z, H) +2 g (A_X Y, A_Z H) - g(A_Y Z, A_X H) + g(A_X Z, A_Y H) \\
		&- \frac{1}{2(n-1)} \Bigg\{g(X, H) \bigg[S^G (Y',Z') \circ \pi
		+\frac{1}{2} \left(g(\nabla_Y N,Z)+g(\nabla_Z N, Y)\right) \\
		&- 2 \sum_{i} g(A_Y X_i, A_Z X_i) - \sum_{j} g(T_{U_j}Y, T_{U_j}Z) \Bigg] \\
		&-g(Y,H) \bigg[S^G (X',Z') \circ \pi + \frac{1}{2} \left(g(\nabla_X N,Z)+ g(\nabla_Z N, X)\right) \\
		&-2 \sum_{i} g(A_X X_i, A_Z X_i)- \sum_{j}g(T_{U_j}X, T_{U_j}Z)  \bigg] \\
		&+g(Y,Z) \bigg[S^G (X',H') \circ \pi + \frac{1}{2} \left(g(\nabla_X N,H)+ g(\nabla_H N, X)\right)  \\
		&-2 \sum_{i} g(A_X X_i, A_H X_i)- \sum_{j}g(T_{U_j}X, T_{U_j}H)  \bigg] \\
		&-g(X,Z) \bigg[S^G (Y',H') \circ \pi + \frac{1}{2} \left(g(\nabla_Y N,H)+ g(\nabla_H N, Y)\right) \\
		&-2 \sum_{i} g(A_Y X_i, A_H X_i)- \sum_{j}g(T_{U_j}Y, T_{U_j}H)  \bigg]\Bigg\},
	\end{align*}
	
	\begin{align*}
		g(W^* (X,Y)Z, V) &=-g((\nabla_Z A)_X Y,V) - g(A_X Y, T_V Z) + g(A_Y Z, T_V X)- g(A_X Z, T_V Y) \\
		&- \frac{1}{2(n-1)} \Bigg\{g(Y, Z) \bigg[g(\nabla_X N,V)
		-\sum_{j} g((\nabla_{U_j} T)_{U_j} X, V) \\
		&+\sum_{i}\left(g((\nabla_{X_i}A)_{X_i} X,V )- 2 g (A_V X_i, T_X X_i) \right) \bigg] \\
		&  -g(X,Z) \bigg[g(\nabla_Y N,V)
		-\sum_{j} g((\nabla_{U_j} T)_{U_j} Y, V) \\
		&+ \sum_{i}\left(g((\nabla_{X_i}A)_{X_i} Y,V  ) - 2 g (A_V X_i, T_Y X_i) \right) \bigg] \Bigg\},
	\end{align*}
	
	\begin{align*}
		g(W^* (X,Y)V, W)&= g((\nabla_V A)_X Y, W)- g((\nabla_W A)_X Y, V)+ g(A_X V, A_Y W)\\
		&- g(A_X W, A_Y V) - g(T_V X, T_W Y) + g(T_W X, T_V Y),
	\end{align*}
	
	\begin{align*}
		&g(W^* (X,V)Y ,W)=g((\nabla_X T)_V W, Y) + g((\nabla_V A)_X Y,W)- g(T_V X, T_W Y) + g(A_X Y, A_Y W)\\
		&-\frac{1}{2(n-1)} \bigg\{ - g(V,W) \Big[ S^G (X', Y') \circ \pi + \frac{1}{2} \left(g(\nabla_X N,Y)+ g(\nabla_Y N,X)\right) \\
		&- 2 \sum_{i} g(A_X X_i, A_Y X_i) -\sum_{j} g(T_{U_j}X, T_{U_j}Y) \Big] - g(X,Y) \Big[\hat{S} (V,W) -g(N,T_V W) \\
		&+ \sum_{i}\left(g((\nabla_{X_i}T)_V W,X_i)+ g(A_{X_i} V, A_{X_i}W)  \right) \Big] \bigg\},
	\end{align*}
	\begin{align*}
		g(W^* (U,V)W ,X)&=g((\nabla_U T)_V W, X) - g((\nabla_V T)_U W, X)\\
		& - \frac{1}{2(n-1)} \bigg\{g(V,W) \Big[g(\nabla_U N,X) -\sum_{j} g(\nabla_{U_j}T)_{U_j}U,X)  \\
		&+ \sum_{i} \{g((\nabla_{X_i}A)_{X_i} X,U)-2g(A_X {X_i}, T_U X_i)\}\Big] \\
		&- g(U,W)  \Big[g(\nabla_V N,X) -\sum_{j} g(\nabla_{U_j}T)_{U_j}V,X)  \\
		&+ \sum_{i} \{g((\nabla_{X_i}A)_{X_i} X,V)-2g(A_X {X_i}, T_V X_i)\}\Big] \bigg\}
	\end{align*}
	and
	\begin{align*}
		&g(W^* (U,V)W ,F)= g(\hat{R} (U,V)W, F) + g(T_U W, T_V F) - g(T_V W, T_U F ) \\
		&- \frac{1}{2(n-1)} \bigg\{ g(F,U) \Big[ \hat{S}(V,W)- g(N, T_V W) \\
		& + \sum_{i}\left(g(\nabla_{X_i}T)_V W, {X_i}+ g(A_{X_i}V, A_{X_i}W)\right)  \Big] \\
		&-g(F,V) \Big[ \hat{S}(U,W) -g(N,T_U W)+ \sum_{i} g((\nabla_{X_i}T)_U W, X_i)+ g(A_{X_i}U, A_{X_i}W) \Big]\\
		&+g(V,W) \Big[\hat{S}(U,F) -g(N, T_U F)+ \sum_{i}\left(g((\nabla_{X_i}T)_U F, X_i)+g(A_{X_i}U,A_{X_i}F)\right)\Big]\\
		&-g(U,W) \Big[\hat{S}(V,F)-g(N,T_V F) + \sum_{i}\left(g((\nabla_{X_i}T)_V F, X_i)+ g(A_{X_i} V, A_{X_i}F)\right)\Big] \bigg\}.
	\end{align*}
\end{theorem}
\begin{proof}
	Let's prove the $6^{th}$ equation of this theorem.
	The following equations are obtained inner production with $F$ to $W^*$ and using \eqref{g1} and \eqref{S1} equations.
	\begin{align*}
		g(W^* (U,V)W,F)&=g(R^M(U,V)W,F)- \frac{1}{2(n-1)} \bigg\{g(F,U)S^M(U,V)\\
		&-g(F,V)S^M(U,W) +g(V,W)S^M(U,F)-g(U,W)S^M(V,F) \bigg\}, \nonumber
	\end{align*}
	\begin{align*}
		g(R^M(U,V)W,F)=g(\hat{R}(U,V)W,F)+g(T_U W, T_V F)- g(T_V W, T_U F)
	\end{align*}
	and
	\begin{align*}
		S^M(U,V)= \hat{S}(U,V)- g(N, T_U V) +  \sum_{i}\left\{g((\nabla_{X_i}T)_U V, X_i)+g(A_{X_i}U, A_{X_i}V) \right\}.
	\end{align*}
	When these equations are substituted in $W^*$, the given result is obtained. Other equations are similarly proved by using Theorem \ref{thm1} and Proposition \ref{pro1}.
\end{proof}

\begin{corollary}
Let $\pi: (M, g) \to(G, g^\prime)$ be a Riemannian submersion, where $(M, g)$ and $(G, g^\prime)$ Riemannian manifolds. If the Riemannian submersion has total umbilical fibres, that  is $N = 0$, then the $M$-projective curvature tensor is given by
	\begin{align*}
		&g(W^* (X,Y)Z, H)\\
		&= g(R^G (X,Y)Z, H) +2 g (A_X Y, A_Z H) - g(A_Y Z, A_X H) + g(A_X Z, A_Y H) \\
		&- \frac{1}{2(n-1)} \Bigg\{g(X, H) \bigg[S^G (Y',Z') \circ \pi
		- 2 \sum_{i} g(A_Y X_i, A_Z X_i) - \sum_{j} g(T_{U_j}Y, T_{U_j}Z) \bigg] \\
		&-g(Y,H) \bigg[S^G (X',Z') \circ \pi
		-2 \sum_{i} g(A_X X_i, A_Z X_i)- \sum_{j}g(T_{U_j}X, T_{U_j}Z)  \bigg] \\
		&+g(Y,Z) \bigg[S^G (X',H') \circ \pi
		-2 \sum_{i} g(A_X X_i, A_H X_i)- \sum_{j}g(T_{U_j}X, T_{U_j}H)  \bigg] \\
		&-g(X,Z) \bigg[S^G (Y',H') \circ \pi
		-2 \sum_{i} g(A_Y X_i, A_H X_i)- \sum_{j}g(T_{U_j}Y, T_{U_j}H)  \bigg]\Bigg\},
	\end{align*}
	\begin{align*}
		g(W^* (X,Y)Z, V) &=-g((\nabla_Z A)_X Y,V) - g(A_X Y, T_V Z) + g(A_Y Z, T_V X)- g(A_X Z, T_V Y) \\
		&- \frac{1}{2(n-1)} \Bigg\{g(Y, Z) \bigg[
		-\sum_{j} g((\nabla_{U_j} T)_{U_j} X, V) \\
		&+\sum_{i}\left(g((\nabla_{X_i}A)_{X_i} X,V )- 2 g (A_V X_i, T_X X_i) \right) \bigg] \\
		&  -g(X,Z) \bigg[
		-\sum_{j} g((\nabla_{U_j} T)_{U_j} Y, V) \\
		&+ \sum_{i}\left(g((\nabla_{X_i}A)_{X_i} Y,V  ) - 2 g (A_V X_i, T_Y X_i) \right) \bigg] \Bigg\},
	\end{align*}
	\begin{align*}
		&g(W^* (X,V)Y ,W)\!=\!g((\nabla_X T)_V W, Y)\!+\!g((\nabla_V A)_X Y,W)\!-\! g(T_V X, T_W Y)\!+\!g(A_X Y, A_Y W)\\
		&-\frac{1}{2(n-1)} \bigg\{ - g(V,W) \Big[ S^G (X', Y') \circ \pi
		- 2 \sum_{i} g(A_X X_i, A_Y X_i) -\sum_{j} g(T_{U_j}X, T_{U_j}Y) \Big]\\
		& - g(X,Y) \Big[\hat{S} (V,W) + \sum_{i}\left(g((\nabla_{X_i}T)_V W,X_i)
		+ g(A_{X_i} V, A_{X_i}W)  \right) \Big] \bigg\},
	\end{align*}
	\begin{align*}
		g(W^* (U,V)W ,X)&=g((\nabla_U T)_V W, X) - g((\nabla_V T)_U W, X)\\
		& - \frac{1}{2(n-1)} \bigg\{g(V,W) \Big[\sum_{j} g(\nabla_{U_j}T)_{U_j}U,X)  \\
		&+ \sum_{i} \{g((\nabla_{X_i}A)_{X_i} X,U)-2g(A_X {X_i}, T_U X_i)\}\Big] \\
		&- g(U,W)  \Big[\sum_{j} g(\nabla_{U_j}T)_{U_j}V,X)  \\
		&+ \sum_{i} \{g((\nabla_{X_i}A)_{X_i} X,V)-2g(A_X {X_i}, T_V X_i)\}\Big] \bigg\}
	\end{align*}
	and
	\begin{align*}
		&g(W^* (U,V)W ,F)\\
		&=g(\hat{R} (U,V)W, F) + g(T_U W, T_V F) - g(T_V W, T_U F ) \\
		&- \frac{1}{2(n-1)} \bigg\{ g(F,U) \Big[ \hat{S}(V,W)
		+ \sum_{i}\left(g(\nabla_{X_i}T)_V W, {X_i}+ g(A_{X_i}V, A_{X_i}W)\right)  \Big] \\
		&-g(F,V) \Big[ \hat{S}(U,W) + \sum_{i} g((\nabla_{X_i}T)_U W, X_i)+ g(A_{X_i}U, A_{X_i}W) \Big]\\
		&+g(V,W) \Big[\hat{S}(U,F) + \sum_{i}\left(g((\nabla_{X_i}T)_U F, X_i)+g(A_{X_i}U,A_{X_i}F)\right)\Big]\\
		&-g(U,W) \Big[\hat{S}(V,F) + \sum_{i}\left(g((\nabla_{X_i}T)_V F, X_i)+ g(A_{X_i} V, A_{X_i}F)\right)\Big] \bigg\}.
	\end{align*}
\end{corollary}

\begin{rem}
	The authors investigate new curvature tensors along Riemannian submersions and obtain some results by using totally umbilical fibres.  Therefore, it will be worth examining new curvature tensors along Riemannian submersions. 
\end{rem}


\begin{ack}
Both authors would like to thank Prof. Dr. Gabriel Eduard Vilcu for some useful comments and questions that help to clarify some statements of the original manuscript.
\end{ack}

\end{document}